\documentclass[11pt,reqno]{amsart}

\usepackage[T1]{fontenc}
\usepackage{lmodern}
\usepackage{microtype}

\usepackage{amsmath,amssymb,amsthm,mathtools}
\usepackage{enumitem}

\usepackage[backend=biber,style=alphabetic,%
maxbibnames=10,giveninits=true,maxalphanames=4]{biblatex}
\usepackage{xcolor}
\usepackage[
  colorlinks=true,
  linkcolor=blue!45!black,
  citecolor=blue!45!black,
  urlcolor=blue!45!black
]{hyperref}

\numberwithin{equation}{section}

\newtheorem{theorem}{Theorem}[section]
\newtheorem{proposition}[theorem]{Proposition}
\newtheorem{lemma}[theorem]{Lemma}
\newtheorem{corollary}[theorem]{Corollary}

\theoremstyle{definition}
\newtheorem{definition}[theorem]{Definition}

\newcommand{\Q}{\mathbb Q}
\newcommand{\Z}{\mathbb Z}
\newcommand{\Zpos}{\mathbb Z_{>0}}
\DeclareMathOperator{\Col}{Col}
\newcommand{\Colun}{\Col_{\mathrm{un}}}

\title[Sum--product growth for fixed-field lattice products]{The sum--product
conjecture holds for fixed-field lattice product constructions}

\author{Harrison Lemley}
\email{hlemley@algebraity.com}

\author{David Chew}
\email{david.c.chew.20@gmail.com}

\date{\today}

\subjclass[2020]{Primary 11B30; Secondary 11J86, 11R27, 11G50, 11P70}
\keywords{sum--product conjecture, sumsets, product sets, number fields,
  unit groups, Weil height, linear forms in logarithms}

\hypersetup{
  pdftitle={The sum-product conjecture holds for fixed-field lattice product
  constructions},
  pdfauthor={Harrison Lemley and David Chew}
}

\begin{document}

\begin{abstract}
We study fixed-field analogues of the lattice product construction
recently used by Bloom, Sawin, Schildkraut, and Zhelezov to disprove the
sum--product conjecture over the real numbers. These sets are of the form $A_n =
P_nG_n \subset F$,
where $F$ is a fixed number field, $P_n$ can be rescaled to a set of
algebraic integers whose conjugates are polynomially bounded in
$|P_n \setminus \{0\}|$, and $G_n$ occupies a positive proportion of a
bounded box of scale $Y_n$ in a fixed positive-rank multiplicative lattice.
We prove that these fixed-field constructions satisfy the sum--product
conjecture:
\[
\max\{|A_n + A_n|,\, |A_nA_n|\} \ge |A_n|^{2 - o(1)}
\]
as $|A_n| \to \infty$, with no restriction on the relative growth of the two
factors. More strongly, any fixed power saving in the product set forces
an asymptotically maximal sumset: there exists $C > 0$ such that, for each
fixed $\delta \in (0, 1]$, if $|A_nA_n| \le |A_n|^{2 - \delta}$ for all
sufficiently large $n$, then $Y_n \to \infty$ and
\[
\frac{|A_n + A_n|}{|A_n|^2}
=
\frac{1}{2} + O_\delta\left(\frac{(\log Y_n)^C}{Y_n}\right)\,.
\]
In particular, the additive-box times unit-box construction of Bloom,
Sawin, Schildkraut, and Zhelezov cannot produce sum--product counterexamples
in any fixed totally real number field, at any relative scales. The proof
combines a bound for factorizations of algebraic integers with a more
general maximal-sumset theorem for coefficient sets whose nonzero-element ratios
have small Weil height.
The latter controls multiplication fibers and combines a
lattice strip estimate with logarithmic separation derived from the
theorems of Yu and Matveev.
\end{abstract}

\maketitle

\section{Introduction and motivation}

\subsection{Background and context}

The sum--product phenomenon, classically associated with Erd\H{o}s and
Szemer\'edi
\cite{ErdosSzemeredi1983}, asks to what extent a finite set can
have both additive and multiplicative structure. If $A$ is a finite subset
of a ring, write
\[
A + A := \{a + a' : a, a' \in A\}\,,
\qquad
AA := \{aa' : a, a' \in A\}\,.
\]
The classical sum--product conjecture predicts that, for every
$\varepsilon > 0$, there exists a constant $c_\varepsilon > 0$ such that every
sufficiently large finite set $A \subset \Z$ satisfies
\[
\max\{|A + A|,\, |AA|\} \ge c_\varepsilon|A|^{2 - \varepsilon}\,.
\]
Thus, if this conjecture holds, a sequence of finite sets of integers should
not simultaneously have additive and multiplicative growth that is
asymptotically subquadratic by a fixed power of the sets' sizes.

A particularly important family in this context is given by sets of the
form
\[
A = PG\,,
\]
where $G$ is a geometric progression and $P$ is an arithmetic
progression or an interval-like set. The geometric factor provides strong
multiplicative structure, while the arithmetic factor supplies many
possible coefficients within each geometric layer. Constructions of this
type appear in the work of Balog and Wooley
\cite{BalogWooley2017} and in subsequent work on sum--product phenomena and
additive-energy estimates
\cite{HansonRudnevShkredovZhelezov2025}.

Recent work of Bloom, Sawin, Schildkraut, and Zhelezov (BSSZ)
\cite{BloomSawinSchildkrautZhelezov2026} gives further motivation for
studying this family. They construct a high-dimensional analogue of the
Balog--Wooley product construction and use it to disprove the sum--product
conjecture over the real numbers. Their construction is carried out in
totally real number fields whose degrees tend to infinity logarithmically with
the size of
the set. Its additive component is obtained from the lattice of algebraic
integers under the real embeddings, while its multiplicative component is
a box in the logarithmic lattice of units. This produces arbitrarily
large finite sets $A \subset \mathbb R$ for which both $|A + A|$ and $|AA|$
are bounded by $|A|^{2 - c}$ for an absolute constant $c > 0$.

BSSZ remark that, because the degree of their fields grows with $|A|$, the
sum--product conjecture may still hold in number fields of bounded degree. In
this direction, we show that their construction, and indeed every lattice
product of the type considered here, satisfies the conjecture in any fixed
number field. This definitively answers the question of whether the construction
of BSSZ could disprove the sum--product conjecture in any fixed number field; we
note that this is not the same as taking sets from potentially distinct number
fields of bounded degree, which is not addressed in the present paper.

We allow
the coefficient set to be substantially more general than an additive
lattice box, and the multiplicative factor to be any positive-density
subset of a bounded exponent box in a fixed finitely generated group.
The two factors may grow at unrelated rates. We show that every family
satisfying the precise hypotheses below has sum--product growth
$|A_n|^{2 - o(1)}$. More strongly, whenever the product set has a fixed power
saving, the sumset is asymptotically maximal.

\subsection{Definition of the sets and main result}

Let $F \subset \mathbb C$ be a fixed number field, let $\mathcal O_F$ be its
ring of integers, and put $d := [F : \Q]$. We first define the two families
used in the main theorem.

\begin{definition}[Coefficient families]\label{def:coefficients}
Let $P_n \subset F$ be finite sets, and put
\[
P_n^* := P_n \setminus \{0\}\,,
\qquad
M_n := |P_n^*|\,,
\qquad
\widehat M_n := \max\{2,\, M_n\}\,.
\]
We call $(P_n)$ an \emph{admissible coefficient family} if $M_n \ge 1$ and
there are fixed constants $C_P \ge 1$ and $\kappa > 0$ such that, for every
$n$, there exists $w_n \in F^\times$ with
\begin{equation}\label{eq:coefficient-hypothesis}
Q_n := w_nP_n^* \subset \mathcal O_F \setminus \{0\}\,,
\qquad
\max_{\substack{q \in Q_n\\ \sigma : F \hookrightarrow \mathbb C}}
|\sigma(q)| \le C_P\widehat M_n^\kappa\,.
\end{equation}
\end{definition}

The condition on $P_n$ is essentially
that, after a common rescaling, all conjugates of its nonzero elements are
bounded by a fixed power of its cardinality. In particular, it permits
irregular subsets of additive lattice boxes, provided those subsets have
cardinality at least a fixed positive power of the box scale.

Let $\mu_F$ denote the finite group of roots of unity in $F$. Fix a finitely
generated subgroup $\Gamma \le F^\times$, and write
\[
\Gamma_{\mathrm{tors}} := \Gamma \cap \mu_F\,,
\qquad
\overline{\Gamma} := \Gamma / \Gamma_{\mathrm{tors}}\,,
\qquad
\rho := \operatorname{rank}_{\Z}\overline{\Gamma}\,.
\]
Assume that $\rho \ge 1$ and fix an isomorphism
\[
\psi : \overline{\Gamma} \longrightarrow \Z^\rho\,.
\]
Let $\pi : \Gamma \to \overline{\Gamma}$ be the quotient map.

\begin{definition}[Multiplicative lattice families]\label{def:lattice}
Let $Y_n \ge 2$ be real numbers. We call a family of finite sets
$G_n \subset \Gamma$ \emph{full-dimensional of scale $Y_n$} if there are
fixed constants $c_G > 0$ and $C_G \ge 1$ such that
\begin{equation}\label{eq:lattice-hypothesis}
\psi(\pi(G_n)) \subset [-C_GY_n, C_GY_n]^\rho\,,
\qquad
|G_n| \ge c_GY_n^\rho
\end{equation}
for every $n$.
\end{definition}

Each exponent vector has at most $|\Gamma_{\mathrm{tors}}|$ preimages, which is
a finite number, so
the first condition also gives $|G_n| \ll Y_n^\rho$. Consequently,
\[
|G_n| \asymp Y_n^\rho\,.
\]
Changing $\psi$ only changes the constants in
\eqref{eq:lattice-hypothesis}.

Throughout this paper, $F$, $\Gamma$, $\psi$, $C_P$, $\kappa$, $c_G$, and
$C_G$ refer to fixed data whenever they occur. All logarithms without
a subscript are natural. We use the standard asymptotic notation $X = O(Y)$,
or equivalently $X \ll Y$, to mean that $|X| \le CY$ for some constant
$C > 0$. We write $X \gg Y$ if $Y \ll X$, and $X \asymp Y$ if both $X \ll Y$
and $X \gg Y$. Unless otherwise specified, implicit constants may depend
on the fixed data, but not on the varying sets or their scales. A subscript
on $O$ or $\ll$ records additional permitted dependence.

Our main result is the following.

\begin{theorem}[Fixed-field obstruction]\label{thm:main}
Let $(P_n)$ be an admissible coefficient family and let $(G_n)$ be a
full-dimensional multiplicative lattice family of scales $Y_n$. Put
\[
A_n := P_nG_n \subset F\,,
\]
and suppose that $|A_n| \to \infty$ as $n \to \infty$. Then
\begin{equation}\label{eq:main-growth}
\max\{|A_n + A_n|,\, |A_nA_n|\} \ge |A_n|^{2 - o(1)}\,.
\end{equation}
More strongly, let $\delta \in (0, 1]$ be fixed. If
\begin{equation}\label{eq:product-saving}
|A_nA_n| \le |A_n|^{2 - \delta}
\end{equation}
for all sufficiently large $n$, then $Y_n \to \infty$ and there exists a
constant $C > 0$ independent of $\delta$ such that
\begin{equation}\label{eq:main-maximal}
\frac{|A_n + A_n|}{|A_n|^2}
=
\frac{1}{2} + O_\delta\left(\frac{(\log Y_n)^C}{Y_n}\right)\,.
\end{equation}
\end{theorem}

Here \eqref{eq:main-growth} means that, for every $\varepsilon > 0$, the
maximum is at least $|A_n|^{2 - \varepsilon}$ for all sufficiently large
$n$. In particular, no family in this class has a fixed power saving in
both its sumset and its product set. Multiplying either factor by an
arbitrary nonzero element of $F$ does not affect these conclusions, as such a
scalar may be absorbed into the coefficient set, where it is removed by
the normalization in \eqref{eq:coefficient-hypothesis}.

To state the application to the construction motivating the paper, let
$F$ be totally real, with embeddings $\sigma_1, \ldots, \sigma_d$. Write
\[
B_F^+(X)
:=
\{\alpha \in \mathcal O_F : |\sigma_i(\alpha)| \le X
\,\forall i\}
\]
and
\[
B_F^\times(Y)
:=
\{u \in \mathcal O_F^\times : \bigl|\log|\sigma_i(u)|\bigr| \le Y
\,\forall i\}\,.
\]
These are the additive and multiplicative boxes used in
\cite[\S\S3--4]{BloomSawinSchildkrautZhelezov2026}; their product
construction takes $P = X + B_F^+(\varepsilon X)$ and $G = B_F^\times(Y)$, with
$X$ integral and $\varepsilon$ a sufficiently small fixed positive constant.

\begin{corollary}[Fixed-field BSSZ boxes]\label{cor:BSSZ}
Let $F$ be a fixed totally real number field of degree $d \ge 2$, and let
$\varepsilon \in (0, 1)$ be fixed. For $X \in \Zpos$ and $Y \ge 2$ real, define
\[
A_{X, Y}
:=
\bigl(X + B_F^+(\varepsilon X)\bigr)B_F^\times(Y)\,.
\]
Along every sequence $(X_n, Y_n)$ for which $|A_{X_n, Y_n}| \to \infty$,
\[
\max\{|A_{X_n, Y_n} + A_{X_n, Y_n}|,\,
|A_{X_n, Y_n}A_{X_n, Y_n}|\}
\ge |A_{X_n, Y_n}|^{2 - o(1)}\,.
\]
If the product sets have a fixed power saving, then their sumsets satisfy
\eqref{eq:main-maximal}. The parameters $X_n$ and $Y_n$ may grow at
arbitrary relative rates.
\end{corollary}

\subsection{Outline of the paper}

Section~2 contains the common framework used throughout the proof. We
define the Weil height and record the height identities needed to prove the main
results.
We then choose exponent generators for $\overline{\Gamma}$ and fix a place at
which the group has nontrivial absolute value. Height counting in a fixed
multiplicative group controls the fibers of the multiplication map
\[
P^* \times G \longrightarrow P^*G\,,
\]
where $P^* = P \setminus \{0\}$.

Section~3 proves the quantitative maximal-sumset estimate used in the main
argument. We introduce the unordered
collision excess and express $|A + A|$ in terms of the number of unordered
pairs and repeated representations. The proof combines a strip-counting
estimate in the exponent lattice with a logarithmic separation argument
derived from Yu's theorem \cite{Yu94} at finite places and Matveev's
theorem \cite{Matveev2000} at archimedean places. The strip estimate
controls pairs whose logarithmic absolute values are close. The separation
estimate shows that two unordered representations of the same sum must
coincide if, in each representation, the logarithmic absolute values of
the two summands are sufficiently far apart.

Section~4 establishes near-quadratic product growth for the coefficient
sets in Definition~\ref{def:coefficients}. The proof counts ideal
divisors and then controls the possible generators of each principal
ideal by height. This shows that a fixed power saving in $|A_nA_n|$
forces the hypotheses of Theorem~\ref{thm:maximal}. Applying that theorem
proves Theorem~\ref{thm:main}. 

Section~5 treats
the fixed-field variant of the BSSZ construction and presents a family of sets
built from arithmetic and geometric
progressions as an additional application of the results. Section~6 gives
statements and disclosures about the
research.

\section{Heights and multiplicative lattices}

\subsection{Places and the Weil height}\label{subsec:height}

We begin by specifying the normalization of the absolute values and
height. Let $M_F$ denote the set of places of $F$. For a nonzero integral
ideal $\mathfrak a \subseteq \mathcal O_F$, define its norm by
\[
N\mathfrak a := |\mathcal O_F / \mathfrak a|\,.
\]
A finite place corresponds to a nonzero prime ideal $\mathfrak q$ of
$\mathcal O_F$.
For $x \in F^\times$, let $v_{\mathfrak q}(x)$ denote the exponent of
$\mathfrak q$ in the principal fractional ideal $(x)$. If $\mathfrak q$
lies above the rational prime $p$, let $e_{\mathfrak q}$ and
$f_{\mathfrak q}$ be its ramification index and residue degree, so that
\[
v_{\mathfrak q}(p) = e_{\mathfrak q}\,,
\qquad
N\mathfrak q = p^{f_{\mathfrak q}}\,.
\]
At this place we use the absolute value
\[
|x|_v := p^{-v_{\mathfrak q}(x) / e_{\mathfrak q}}
\]
and the local weight $n_v := e_{\mathfrak q}f_{\mathfrak q}$. At a real
place represented by an embedding $\sigma : F \hookrightarrow \mathbb R$,
we put $|x|_v := |\sigma(x)|$ and $n_v := 1$. At a complex place, represented
by a conjugate pair of nonreal embeddings, we use the ordinary complex
modulus and put $n_v := 2$. All these absolute values
satisfy the triangle inequality, and at finite places they satisfy the
stronger inequality
\[
|x + y|_v \le \max\{|x|_v,\, |y|_v\}\,.
\]
A nonzero element of $F$ is an algebraic integer precisely when
$v_{\mathfrak q}(x) \ge 0$ for every $\mathfrak q$, and it is an algebraic
unit precisely when all these valuations are zero. Thus integrality is
equivalent to $|x|_v \le 1$ at every finite place.

With this normalization the product formula is
\begin{equation}\label{eq:product-formula}
\sum_{v \in M_F}n_v\log|x|_v = 0
\qquad(x \in F^\times)\,.
\end{equation}
Indeed, the archimedean contribution is
$\log|N_{F / \Q}(x)|$, whereas the finite contribution is
\[
-\sum_{\mathfrak q}v_{\mathfrak q}(x)\log N\mathfrak q
= -\log N((x))
= -\log|N_{F / \Q}(x)|\,.
\]
Here the norm of a nonzero fractional ideal is obtained by multiplicatively
extending the norm of integral ideals. For a nonzero algebraic integer
$x$, the equality $N((x)) = |N_{F / \Q}(x)|$ follows by expressing
multiplication by $x$ in an integral basis: the absolute determinant is
both the index $[\mathcal O_F : x\mathcal O_F]$ and the absolute value of
the field norm. Clearing a rational integer denominator proves the same
identity for $x \in F^\times$. Only finitely many terms in these sums are
nonzero.

\begin{definition}[Absolute logarithmic Weil height]\label{def:height}
For $x \in F$, define
\begin{equation}\label{eq:height-definition}
h(x) := \frac{1}{d}\sum_{v \in M_F}n_v\log\max\{1,\, |x|_v\}\,.
\end{equation}
In particular, $h(0) = 0$.
\end{definition}

The height records the size of an algebraic number at all places,
including the contribution of denominators. For example, if $a, b \in \Z$
are coprime and $b > 0$, then
\[
h(a / b) = \log\max\{|a|,\, b\}\,.
\]
The normalization by $d$ makes the height independent of the number field
used to compute it. More precisely, if $E / F$ is finite and $x \in F$, the
places of $E$ above a place $v$ have weights summing to $[E : F]n_v$, and
their absolute values agree with $| \cdot |_v$ on $F$. Dividing by
$[E : \Q] = [E : F]d$ therefore gives the same value. An embedding of $F$ into
$\mathbb C$ transports its places and preserves their weights, so the
height is also unchanged by conjugation. This is the standard absolute
logarithmic height; see \cite[Chapter~1]{BombieriGubler2006}.

We record the properties needed in the proof. They are standard results and will
not be proved here.

\begin{lemma}[Height identities]\label{lem:height-identities}
For $x, y \in F^\times$, $b \in \Z$, and $\xi \in \mu_F$, we have
\begin{align*}
h(x^{-1})&= h(x)\,,
& h(x^b)&= |b|h(x)\,, \\
h(xy)&\le h(x) + h(y)\,,
& h(x / y)&\le h(x) + h(y)\,, \\
h(x + y)&\le h(x) + h(y) + \log 2\,,
& h(\xi x)&= h(x)\,.
\end{align*}
For every place $v$,
\begin{equation}\label{eq:local-height-bound}
\bigl|\log|x|_v\bigr|
\le \frac{d}{n_v}h(x)
\le d\,h(x)\,.
\end{equation}
Moreover,
\begin{equation}\label{eq:height-absolute-sum}
2d\,h(x) = \sum_{v \in M_F}n_v\bigl|\log|x|_v\bigr|\,.
\end{equation}
If $x$ is an algebraic integer, then
\begin{equation}\label{eq:integral-height}
h(x) = \frac{1}{d}\sum_{\sigma : F \hookrightarrow \mathbb C}
\log\max\{1,\, |\sigma(x)|\}\,,
\end{equation}
where the sum is over all $d$ embeddings. In particular, if all conjugates
of $x$ have modulus at most $T \ge 1$, then $h(x) \le \log T$.
Finally, $h(x) = 0$ for $x \ne 0$ if and only if $x$ is a root of unity.
\end{lemma}

For $z \in \mathbb C^\times$, let $\operatorname{Arg}z$ be the unique
$\theta \in (-\pi, \pi]$ such that $z = |z|e^{i\theta}$, and define the
principal value of the complex logarithm by
\[
\operatorname{Log}z := \log|z| + i\operatorname{Arg}z\,.
\]
For every $x \in F^\times$ and every embedding
$\sigma : F \hookrightarrow \mathbb C$, we have
\begin{equation}\label{eq:complex-log-height}
\bigl|\operatorname{Log}\sigma(x)\bigr|
\le \bigl|\log|\sigma(x)|\bigr| + \pi
\le d\,h(x) + \pi\,.
\end{equation}
This is the height bound for logarithms used in Section~3.

For a finite $P \subset F$ with $P^* := P \setminus \{0\} \ne \varnothing$, we
define the ratio height to be
\[
R(P) := \max_{m, m' \in P^*}h(m / m')\,.
\]
This quantity is unchanged by multiplying $P$ by any nonzero scalar.

\begin{lemma}[Ratio heights of admissible
coefficients]\label{lem:coefficient-height}
For the family in Definition~\ref{def:coefficients},
\[
R(P_n) \le 2\log C_P + 2\kappa\log\widehat M_n
\ll \log\widehat M_n\,.
\]
\end{lemma}

\begin{proof}
For $q \in Q_n$, \eqref{eq:integral-height} and
\eqref{eq:coefficient-hypothesis} give
\[
h(q) \le \log C_P + \kappa\log\widehat M_n\,.
\]
If $m, m' \in P_n^*$, put $q = w_nm$ and $q' = w_nm'$. Then $m / m' = q / q'$,
so Lemma~\ref{lem:height-identities} gives
\[
h(m / m') = h(q / q') \le h(q) + h(q')
\le 2\log C_P + 2\kappa\log\widehat M_n\,.
\]
Taking the maximum proves the assertion.
\end{proof}

\subsection{Exponent notation and a nontrivial place}

Choose $\gamma_1, \ldots, \gamma_\rho \in \Gamma$ whose images under
$\psi \circ \pi$ are the standard basis vectors of $\Z^\rho$. Every
$g \in \Gamma$ has a unique expression
\[
g = \xi\prod_{i = 1}^\rho\gamma_i^{z_i}\,,
\qquad
\xi \in \Gamma_{\mathrm{tors}}\,,
\qquad
(z_1, \ldots, z_\rho) := \psi(\pi(g))\,.
\]
If $G \subset \Gamma$ and $Y \ge 2$ satisfy
\[
\psi(\pi(G)) \subset [-C_GY, C_GY]^\rho\,,
\]
then the exponents of every $g \in G$ satisfy
$|z_i| \le C_GY$ for $1 \le i \le \rho$.
Thus each exponent coordinate ranges over an interval of length $O(Y)$.

\begin{lemma}[A nontrivial local absolute value]\label{lem:nontrivial-place}
There is a place $v$ of $F$ at which $|\gamma_i|_v \ne 1$ for at least one
$i$. At such a place, the map
\[
\lambda_v(x) := \log|x|_v
\qquad(x \in F^\times)
\]
factors on $\Gamma$ through a nonzero real linear functional
\[
\ell_v : \mathbb R^\rho \longrightarrow \mathbb R\,,
\qquad
\ell_v(\psi(\pi(g))) = \lambda_v(g)\,.
\]
\end{lemma}

\begin{proof}
Each $\gamma_i$ has infinite order modulo torsion, so it is not a root
of unity. If its absolute value were $1$ at every place, its height would
be zero, contrary to Lemma~\ref{lem:height-identities}. Thus an appropriate
place exists. Since roots of unity have absolute value $1$ at every
place, we may define
\[
\ell_v(x_1, \ldots, x_\rho)
:= \sum_{i = 1}^\rho x_i\log|\gamma_i|_v\,.
\]
At least one coefficient is nonzero, and the claimed identity follows
from the expression for $g$ above.
\end{proof}

Fix such a place $v$ for the rest of the paper.

\subsection{Height counting in a multiplicative group}

We state the height-counting estimate for an arbitrary fixed finitely
generated subgroup, since it will also be used for the full unit group
in Section~4.

\begin{lemma}[Height counting]\label{lem:Gamma-height-count}
Let $\Gamma_0 \le F^\times$ be a fixed finitely generated subgroup, and let
$r$ be its rank modulo torsion. For every $T \ge 0$,
\[
|\{\gamma \in \Gamma_0 : h(\gamma) \le T\}|
\ll_{F, \Gamma_0} (T + 1)^r\,.
\]
\end{lemma}

\begin{proof}
Let $S$ consist of all archimedean places together with the finite places
at which some member of a fixed generating set for $\Gamma_0$ has
nonzero valuation. This is a fixed finite set. Define the group of
$S$-units by
\[
\mathcal O_{F, S}^{\times }
:= \{\alpha \in F^\times : |\alpha|_w = 1\text{ for every }w \notin S\}\,.
\]
Then $\Gamma_0 \subseteq \mathcal O_{F, S}^{\times }$.

The $S$-unit theorem states that
$\mathcal O_{F, S}^{\times } / \mu_F$ is free abelian of rank $|S| - 1$, and
that
the logarithmic map
\[
\lambda_S : \mathcal O_{F, S}^{\times } / \mu_F \longrightarrow \mathbb R^S\,,
\qquad
\lambda_S(\alpha\mu_F)
:= \bigl(n_w\log|\alpha|_w\bigr)_{w \in S}
\]
has image a full lattice in the hyperplane
\[
\left\{(x_w)_{w \in S} : \sum_{w \in S}x_w = 0\right\}\,.
\]
Here a lattice is the set of integer combinations of a real basis of its
span. We use this form of the theorem as in
\cite[Chapter~3]{Narkiewicz2004}. The hyperplane condition is exactly the
product formula, and the kernel before quotienting is $\mu_F$, by
Lemma~\ref{lem:height-identities}.

The image of
$\Gamma_0 / (\Gamma_0 \cap \mu_F)$ is a subgroup of this lattice. It is
therefore a rank-$r$ lattice in its real span. To see the rank assertion
explicitly, choose lattice coordinates for the full $S$-unit lattice.
The subgroup then lies in an integer lattice, where the rank of an
integer matrix over $\Q$ agrees with its rank over $\mathbb R$.

For an $S$-unit $\gamma$, \eqref{eq:height-absolute-sum} gives
\[
2d\,h(\gamma) = \|\lambda_S(\gamma\mu_F)\|_1\,.
\]
Thus bounded height is a bounded-norm condition in a fixed lattice.
Choose a basis $\mathbf b_1, \ldots, \mathbf b_r$ of the subgroup lattice.
The inverse of the coordinate map on its real span is bounded, so
\[
\left\|\sum_{i = 1}^r z_i\mathbf b_i\right\|_1 \le 2dT
\quad \Longrightarrow \quad
|z_i| \ll T
\qquad(1 \le i \le r)\,.
\]
There are $O((T + 1)^r)$ possible integer vectors. Each logarithmic vector
has at most $|\Gamma_0 \cap \mu_F|$ preimages in $\Gamma_0$, which proves
the bound. If $r = 0$, the group is finite and the same assertion holds.
\end{proof}

\subsection{Fibers of the multiplication map}

\begin{lemma}[Multiplication fibers]\label{lem:A-size}
Let $P \subset F$ be finite with $P^* \ne \varnothing$, and let
$\varnothing \ne G \subset \Gamma$ be finite. Let $A := PG$ and $A^* := P^*G$.
Then every
fiber of
\[
P^* \times G \longrightarrow A^*\,,
\qquad
(m, g) \longmapsto mg
\]
has size $O((1 + R(P))^\rho)$. Consequently,
\[
\frac{|P^*||G|}{(1 + R(P))^\rho}
\ll |A^*| \le |P^*||G|\,.
\]
\end{lemma}

\begin{proof}
Fix $x \in A^*$ and one factorization $x = m_0g_0$ with
$m_0 \in P^*$ and $g_0 \in G$. If $x = mg$ is any other factorization,
then
\[
\frac{g}{g_0} = \frac{m_0}{m} \in \Gamma\,,
\qquad
h(g / g_0) = h(m_0 / m) \le R(P)\,.
\]
By Lemma~\ref{lem:Gamma-height-count}, there are at most
$O((1 + R(P))^\rho)$ possibilities for $g / g_0$. Each such ratio determines
$g$ and then $m = x / g$ uniquely. Thus every fiber has the stated size.
Dividing the number $|P^*||G|$ of input pairs by this bound proves the
lower estimate. The upper estimate is immediate.
\end{proof}

Let $P$, $G$, $A$, and $A^*$ be as in Lemma~\ref{lem:A-size}, and let
$Y \ge 2$. Suppose that $C_h > 0$, $c_G > 0$, and $C_G \ge 1$ are fixed, and
that
\[
R(P) \le C_h\log Y\,,
\qquad
\psi(\pi(G)) \subset [-C_GY, C_GY]^\rho\,,
\qquad
|G| \ge c_GY^\rho\,.
\]
Put $M := |P^*|$. Then $|G| \asymp Y^\rho$ and the preceding lemma gives
\begin{equation}\label{eq:low-height-size}
\frac{MY^\rho}{(\log Y)^\rho}
\ll |A^*| \le |A| \ll MY^\rho\,.
\end{equation}
The upper bound for $A$ also accounts for a possible zero element, since
$M \ge 1$. In addition, choosing any $m_0 \in P^*$ gives
$m_0G \subseteq A$, so
\begin{equation}\label{eq:A-at-least-G}
|A| \ge |G| \gg Y^\rho\,.
\end{equation}

\section{Asymptotically maximal sumsets from logarithmic separation}

The following result is the input used later in
the proof of Theorem~\ref{thm:main} to force an asymptotically maximal sumset
when product growth is small.

\begin{theorem}[Asymptotically maximal sumsets for low-ratio-height
products]\label{thm:maximal}
Fix $C_h > 0$, $c_G > 0$, and $C_G \ge 1$. Let $Y \ge 2$, let $P \subset F$ be
finite with $P^* := P \setminus \{0\} \ne \varnothing$, and let
$G \subset \Gamma$ be finite. Suppose that
\[
R(P) \le C_h\log Y\,,
\qquad
\psi(\pi(G)) \subset [-C_GY, C_GY]^\rho\,,
\qquad
|G| \ge c_GY^\rho\,.
\]
Then, for $A := PG$, there exists a constant $C > 0$ such that
\[
\frac{|A + A|}{|A|^2}
=
\frac{1}{2} + O\left(\frac{(\log Y)^C}{Y}\right)\,.
\]
The constants are uniform in $P$, $G$, and $Y$, and may depend on
$F$, $\Gamma$, $\psi$, $C_h$, $c_G$, and $C_G$.
\end{theorem}

We now prove the theorem. Throughout this section, $P$, $G$, and $Y$
satisfy its hypotheses, and we write
\[
P^* := P \setminus \{0\}\,,
\qquad
A := PG\,,
\qquad
A^* := P^*G = A \setminus \{0\}\,,
\qquad
M := |P^*|\,.
\]
The argument combines a $\rho$-dimensional strip estimate with a
separation lemma at the place $v$ fixed in
Lemma~\ref{lem:nontrivial-place}. The separation input is a consequence
of Yu's estimates for linear forms in $p$-adic logarithms at finite places
and Matveev's estimates for linear forms in complex logarithms at
archimedean places.

\subsection{Unordered collisions}

We begin by fixing a standard definition of unordered representations.

\begin{definition}\label{def:unrep}
Let $A$ be a finite subset of a commutative group. For each $t$, define the
unordered representation function
\[
u_A(t)
:=
\bigl|\{\{a, b\} : a, b \in A, \ a + b = t\}\bigr|\,,
\]
where $\{a, b\}$ denotes a two-element multiset, so repetitions are allowed.
Define the \emph{unordered collision excess}
\[
\Colun(A) := \sum_t \max\{u_A(t) - 1,\, 0\}
=
\sum_{t : \,u_A(t) \ge 1}\bigl(u_A(t) - 1\bigr)\,.
\]
\end{definition}

\begin{lemma}[Collision identity]\label{lem:collision-identity}
For any finite subset $A \ne \varnothing$ of a commutative group,
\[
|A + A|
=
\frac{|A|^2 + |A|}{2} - \Colun(A)\,.
\]
Consequently,
\[
\frac{|A + A|}{|A|^2}
=
\frac{1}{2} + \frac{1}{2|A|} - \frac{\Colun(A)}{|A|^2}\,.
\]
\end{lemma}

\begin{proof}
The number of two-element multisets drawn from $A$ is $(|A|^2 + |A|) / 2$.
Partition these multisets by their sum. For each $t$, there are $u_A(t)$
unordered pairs with sum $t$, and these $u_A(t)$ pairs contribute exactly
$1$ distinct element to $A + A$ if $u_A(t) \ge 1$, and contribute nothing
otherwise. Hence
\[
|A + A| = \bigl|\{t : u_A(t) \ge 1\}\bigr|\,.
\]
Also,
\[
\sum_t u_A(t) = \frac{|A|^2 + |A|}{2}\,.
\]
Therefore
\begin{align*}
\Colun(A)
&=
\sum_{t : \,u_A(t) \ge 1}\bigl(u_A(t) - 1\bigr) \\
&=
\sum_t u_A(t) - \bigl|\{t : u_A(t) \ge 1\}\bigr| \\
&=
\frac{|A|^2 + |A|}{2} - |A + A|\,,
\end{align*}
which proves the first identity. Dividing by $|A|^2$ yields the second.
\end{proof}

We also need to remove the possible zero element before taking
multiplicative ratios and logarithmic absolute values.

\begin{lemma}[Adjoining zero]\label{lem:adjoin-zero}
Let $A$ be a finite subset of a commutative group with $0 \notin A$, and put
$B := A \cup \{0\}$. Then
\[
\Colun(B) \le \Colun(A) + |A| + 1\,.
\]
\end{lemma}

\begin{proof}
Passing from $A$ to $B$ adds exactly the $|A| + 1$ unordered pairs
\[
\{0, a\}\,, \qquad(a \in A)
\]
and $\{0, 0\}$. Adding one unordered representation can increase the
collision excess by at most $1$. Hence adjoining all of these
representations increases the collision excess by at most $|A| + 1$.
\end{proof}

\subsection{A logarithmic separation input}

We derive the following separation estimate from the results of Yu and
Matveev for linear forms in logarithms.

\begin{corollary}[Fixed-field logarithmic
separation]\label{cor:local-separation}
Let $\mathfrak v$ be a fixed place of $F$, and let $s \ge 1$ be fixed. For
every fixed constant $C_0 > 0$ there exists a constant
$C_{\mathrm{loc}} > 0$, depending only on
$F$, $\mathfrak v$, $s$, and $C_0$, with the following property.

Let $Y \ge 2$, let $\alpha_1, \ldots, \alpha_s \in F^\times$, and let
$b_1, \ldots, b_s \in \Z$. Suppose that
\[
\Theta := \alpha_1^{b_1}\cdots\alpha_s^{b_s} - 1 \ne 0\,,
\]
and that
\[
h(\alpha_i) \le C_0\log Y\,,
\qquad
|b_i| \le C_0Y
\qquad(1 \le i \le s)\,.
\]
Then
\[
|\Theta|_{\mathfrak v}
\ge
\exp\left(-C_{\mathrm{loc}}(\log Y)^{s + 3}\right)\,.
\]
\end{corollary}

\begin{proof}
Choose a fixed nonzero algebraic integer $\zeta$ such that $F = \Q(\zeta)$.
First suppose that $\mathfrak v$ is a finite place, corresponding
to a prime ideal $\mathfrak p$ of $\mathcal O_F$. Let $p$ be the rational
prime below $\mathfrak p$, and let $e_{\mathfrak p}$ and
$f_{\mathfrak p}$ be the ramification index and residue degree of
$\mathfrak p$ over $p$, respectively. Set $\nu := s + 1$; this is the
parameter $n$ in Yu's theorem.
To make the field generated by the bases in Yu's theorem equal to the fixed
field $F$, define
\[
\beta_i := \alpha_i\,,
\qquad
e_i := b_i
\qquad(1 \le i \le s)\,,
\]
and append the fixed base and zero exponent
\[
\beta_\nu := \zeta\,,
\qquad
e_\nu := 0\,.
\]
Then
\[
\Q(\beta_1, \ldots, \beta_\nu) = \Q(\zeta) = F
\]
and
\[
\beta_1^{e_1}\cdots\beta_\nu^{e_\nu} - 1 = \Theta \ne 0\,.
\]
In particular, the integers $e_1, \ldots, e_\nu$ are not all zero. Since
$\nu = s + 1 \ge 2$, Yu's Theorem~1'
\cite[Theorem~1', pp.~246--247]{Yu94} applies with Yu's parameters
\[
K_0 = F\,,
\qquad
D_0 = d\,,
\qquad
\mathfrak p_0 = \mathfrak p\,,
\qquad
f_0 = f_{\mathfrak p}\,,
\qquad
e_0 = e_{\mathfrak p}\,.
\]
Yu's additive valuation $\operatorname{ord}_{\mathfrak p_0}$ is precisely
the valuation $v_{\mathfrak p}$ used in this paper.

We next verify Yu's height parameters. For every $\alpha \in F^\times$,
Lemma~\ref{lem:height-identities} gives
\[
\bigl|\log|\alpha|\bigr| \le d\,h(\alpha)\,.
\]
By \eqref{eq:complex-log-height}, we also have
\[
|\operatorname{Log}\alpha|
\le
\bigl|\log|\alpha|\bigr| + \pi
\le
d\,h(\alpha) + \pi\,.
\]
Yu's condition~(0.21)*, written using this logarithm convention,
requires ordered real numbers
$V_1 \le \cdots \le V_\nu$ satisfying
\[
V_i \ge
\max\left\{
 h(\beta_i),\,
 \frac{|\operatorname{Log}\beta_i|}{10d},\,
 \frac{1}{2}\left(\frac{f_{\mathfrak p}}{d}\right)^2\log p
\right\}
\qquad(1 \le i \le \nu)\,.
\]
Since $Y \ge 2$, the hypotheses on the $\alpha_i$, the fixed height and
logarithm of $\zeta$, and the preceding inequality allow us to choose
\[
V_1 = \cdots = V_\nu =: V_* := C_1\log Y\,,
\]
where $C_1 > 0$ depends only on $F$, $\mathfrak p$, and $C_0$.
The ordering condition is then automatic. Moreover, Yu's auxiliary
quantity $V$ is either $V_{\nu - 1}$ or $V_\nu$, so in either case
\[
V = V_*\,.
\]

For Yu's exponent parameters, take
\[
B := \max\{3,\, |e_1|,\, \ldots,\, |e_\nu|\}
=
\max\{3,\, |b_1|,\, \ldots,\, |b_s|\}
\]
and $B_i := B$ for $1 \le i \le \nu$. Then
\[
B \le \max\{3,\, C_0Y\} \ll_{C_0} Y\,,
\]
and Yu's conditions
\[
B \ge \max\{|e_1|,\, \ldots,\, |e_\nu|,\, 3\}\,,
\qquad
\max\{1,\, |e_i|\} \le B_i \le B
\]
are satisfied.

With these choices, Theorem~1' gives
\begin{align*}
v_{\mathfrak p}(\Theta)
< {}&
C'(\nu + 1)^{2\nu + 4}
\frac{p^{f_{\mathfrak p}} - 1}{(\log p)^{\nu + 2}}
\left(\frac{d}{f_{\mathfrak p}}\right)^{2\nu + 2}
V_*^\nu\log(4d^2B)\\
&\quad \times
\max\left\{
\log\bigl(2^{13} \cdot 3\nu(\nu + 1)d^3V_*\bigr),\,
\frac{f_{\mathfrak p}\log p}{\nu}
\right\}\,,
\end{align*}
where $C'$ is Yu's explicit constant, depending only on
$\nu$, $p$, and $e_{\mathfrak p}$. All factors outside
$V_*^\nu\log(4d^2B)$ and the final maximum are fixed. Since
$V_* \asymp \log Y$ and $B \ll Y$, each of the two remaining logarithmic
factors is $O(\log Y)$. Hence, for a constant $C_4 > 0$ depending only on
$F$, $\mathfrak p$, $s$, and $C_0$,
\[
v_{\mathfrak p}(\Theta)
\le C_4(\log Y)^{\nu + 2}
=
C_4(\log Y)^{s + 3}\,.
\]
Since
\[
\log|\Theta|_{\mathfrak v}
=
-\frac{\log p}{e_{\mathfrak p}}v_{\mathfrak p}(\Theta)\,,
\]
absorbing the fixed normalization factor into $C_{\mathrm{loc}}$ proves
the required bound at a finite place.

Now suppose that $\mathfrak v$ is archimedean, represented by an embedding
$\sigma : F \hookrightarrow \mathbb C$, and put
\[
z := \prod_{i = 1}^s\sigma(\alpha_i)^{b_i}\,.
\]
Then $z \ne 1$ and $|\Theta|_{\mathfrak v} = |z - 1|$. If $|z - 1| \ge 1 / 2$,
the
desired bound follows by taking $C_{\mathrm{loc}}$ sufficiently large. We may
therefore assume that $|z - 1| < 1 / 2$.

Set
\[
\Lambda := \operatorname{Log}z \ne 0\,.
\]
The sum of the logarithms of the factors differs from $\Lambda$ by an
integer multiple of $2\pi i$. Hence, for some $q \in \Z$,
\[
\Lambda
=
\sum_{i = 1}^s b_i\operatorname{Log}\sigma(\alpha_i)
- 2q\operatorname{Log}(-1)\,.
\]
Comparing imaginary parts gives
\[
2|q| \le \sum_{i = 1}^s|b_i| + 1 \ll Y\,.
\]
Thus every integer coefficient in this linear form is $O(Y)$.
Moreover, invariance of the absolute height under embeddings and the
height inequality used above give
\[
\bigl|\operatorname{Log}\sigma(\alpha_i)\bigr|
\le d\,h(\alpha_i) + \pi \ll \log Y\,.
\]
The additional base $-1$ has height $0$ and logarithm $\pi i$.

Discard terms with zero coefficient or zero logarithm, and let $t_0$
be the number of remaining terms. Since $\Lambda \ne 0$, we have
$1 \le t_0 \le s + 1$. If $t_0 = 1$, append the term $0 \cdot \log 2$ to the
sum. This leaves $\Lambda$
unchanged and ensures that the application of Matveev's
Corollary~2.3 has at least two logarithms. Reindex the resulting
terms and denote their number by $t$. Since $s \ge 1$, we have
$2 \le t \le s + 1$. Thus
\[
\Lambda = \sum_{j = 1}^t c_j\operatorname{Log}\eta_j\,,
\qquad
\eta_j \in \sigma(F)^\times\,,
\qquad
c_j \in \Z\,,
\]
with the logarithms chosen above and each
$\operatorname{Log}\eta_j \ne 0$. For Matveev's height parameters, denoted
here by $\mathcal A_j$, take
\[
\mathcal A_1 = \cdots = \mathcal A_t =: \mathcal A_* := C_2\log Y\,,
\]
where $C_2 > 0$ is a fixed constant large enough that
\[
\mathcal A_j \ge \max\{d\,h(\eta_j),\, |\operatorname{Log}\eta_j|,\, 0.16\}
\qquad(1 \le j \le t)\,.
\]
Since all $\mathcal A_j$ are equal, Matveev's exponent parameter is
\[
B = \max\{1,\, |c_1|,\, \ldots,\, |c_t|\} \ll Y\,.
\]
Matveev's estimate \cite[Corollary~2.3, p.~1219]{Matveev2000}, which does
not require linear independence of the logarithms, now gives
\[
\log|\Lambda|
>
-C_{\mathrm{Mat}} d^2 \mathcal A_*^t\log(ed)\log(eB)\,,
\]
where $C_{\mathrm{Mat}} > 0$ may be chosen to depend only on $t$, by taking
the larger of the real and complex constants. Since $t \le s + 1$ and
$Y \ge 2$, it follows that
\[
|\Lambda| \ge \exp\left(-C_3(\log Y)^{s + 2}\right)
\]
for a fixed constant $C_3 > 0$.

Finally, the power series for $\operatorname{Log}(1 + w)$ gives
\[
|\Lambda|
\le \frac{|z - 1|}{1 - |z - 1|}
\le 2|z - 1|\,.
\]
Consequently,
\[
|\Theta|_{\mathfrak v} = |z - 1|
\ge \frac{1}{2}\exp\left(-C_3(\log Y)^{s + 2}\right)
\ge \exp\left(-C_{\mathrm{loc}}(\log Y)^{s + 3}\right)\,,
\]
after increasing $C_{\mathrm{loc}}$. All constants depend only on the data
listed in the statement. This proves the corollary.
\end{proof}

\begin{lemma}[Separation of low-height multiplicative
translates]\label{lem:bounded-separation}
There exists a constant $C_{\mathrm{sep}} > 0$, depending only on the fixed
data of Theorem~\ref{thm:maximal}, such that, for
\[
m, m' \in P^*\,,
\qquad
g, g' \in G\,,
\qquad
mg \ne m'g'\,,
\]
we have
\[
\left|\frac{mg}{m'g'} - 1\right|_v
\ge
\exp\left(-C_{\mathrm{sep}}(\log Y)^{\rho + 4}\right)\,.
\]
\end{lemma}

\begin{proof}
Write
\[
g = \xi\prod_{i = 1}^\rho\gamma_i^{z_i}\,,
\qquad
g' = \xi'\prod_{i = 1}^\rho\gamma_i^{z_i'}\,,
\]
where $\xi, \xi' \in \Gamma_{\mathrm{tors}}$ and
$|z_i|, |z_i'| \le C_GY$. Then
\[
\frac{mg}{m'g'}
=
\left(\frac{m}{m'}\frac{\xi}{\xi'}\right)
\prod_{i = 1}^\rho\gamma_i^{b_i}\,,
\qquad
b_i := z_i - z_i'\,,
\qquad
|b_i| \le 2C_GY\,.
\]
Set
\[
\alpha_0 := \frac{m}{m'}\frac{\xi}{\xi'}\,,
\qquad
c_0 := 1\,,
\]
and, for $1 \le i \le \rho$, set $\alpha_i := \gamma_i$ and $c_i := b_i$.
Since $\xi / \xi'$ is a root of unity,
Lemma~\ref{lem:height-identities} gives
\[
h(\alpha_0) = h(m / m') \le C_h\log Y\,.
\]
The remaining bases are fixed, so $h(\alpha_i) = O(1)$. Choose a fixed
constant
\[
C_0 \ge \max\left\{
1,\, C_h,\, 2C_G,\,
\frac{h(\gamma_1)}{\log 2},\, \ldots,\,
\frac{h(\gamma_\rho)}{\log 2}
\right\}\,.
\]
All bases then have height at most $C_0\log Y$, and all exponents have
absolute value at most $C_0Y$. The hypothesis $mg \ne m'g'$ says precisely
that
\[
\alpha_0^{c_0}\alpha_1^{c_1}\cdots\alpha_\rho^{c_\rho} - 1 \ne 0\,.
\]
Applying Corollary~\ref{cor:local-separation} with $s = \rho + 1$ proves the
bound.
\end{proof}

\subsection{Strip estimate}

We next prove the strip estimate used in the proof of
Proposition~\ref{prop:collision-bound}.

\begin{lemma}[Strip bound]\label{lem:strip-bound}
For every interval $J \subset \mathbb R$ of finite length,
\[
\left|\{g \in G : \lambda_v(g) \in J\}\right|
\ll
Y^{\rho - 1}\bigl(\operatorname{length}(J) + 1\bigr)\,.
\]
\end{lemma}

\begin{proof}
Under the fixed identification $\psi : \overline{\Gamma} \to \Z^\rho$, the
image of $G$ lies in a box of side lengths $O(Y)$, and each exponent
vector has at most $|\Gamma_{\mathrm{tors}}|$ preimages. Moreover,
\[
\lambda_v(g) = \ell_v(\psi(\pi(g)))\,,
\]
where $\ell_v : \mathbb R^\rho \to \mathbb R$ is a fixed nonzero real linear
functional. Put $L_i := \log|\gamma_i|_v$ for $1 \le i \le \rho$. Then
\[
\ell_v(x_1, \ldots, x_\rho) = L_1x_1 + \cdots + L_\rho x_\rho\,,
\]
where the $L_i$ are fixed real numbers, not all zero. After relabeling the
coordinates, we may assume that $L_1 \ne 0$. For each fixed choice of
$x_2, \ldots, x_\rho$, the condition $\ell_v(\mathbf x) \in J$ restricts
$x_1$ to an interval of length $\operatorname{length}(J) / |L_1|$. This
interval contains at most
\[
\frac{\operatorname{length}(J)}{|L_1|} + 1
\ll \operatorname{length}(J) + 1
\]
integers. The remaining $\rho - 1$ coordinates have $O(Y^{\rho - 1})$
possibilities. Hence the number of exponent vectors satisfying the
restriction is
\[
\ll Y^{\rho - 1}\bigl(\operatorname{length}(J) + 1\bigr)\,.
\]
Multiplying by the fixed bound for the torsion fibers gives the result.
\end{proof}

\subsection{The collision bound}

\begin{proposition}[Collision bound]\label{prop:collision-bound}
Under the hypotheses of Theorem~\ref{thm:maximal},
\[
\Colun(A)
\ll
M^2Y^{2\rho - 1}(\log Y)^{\rho + 4}\,.
\]
\end{proposition}

\begin{proof}
For each $x \in A^*$, fix once and for all one factorization
\[
x = m(x)g(x)\,,
\qquad
m(x) \in P^*\,,
\qquad
g(x) \in G\,.
\]

Let $C_{\mathrm{sep}}$ be the constant from
Lemma~\ref{lem:bounded-separation}, and define
\[
H_Y := C_{\mathrm{sep}}(\log Y)^{\rho + 4} + \log 4\,.
\]
We call an unordered representation $\{x, y\}$ of the sum
\[
x + y
\qquad(x, y \in A^*)
\]
bad if
\[
|\lambda_v(x) - \lambda_v(y)| \le H_Y\,,
\]
and good otherwise.

We first count bad representations. Fix choices $m, m' \in P^*$ and
$g \in G$. If $g' \in G$ satisfies
\[
|\lambda_v(mg) - \lambda_v(m'g')| \le H_Y\,,
\]
then $\lambda_v(g')$ belongs to an interval of length $2H_Y$. By
Lemma~\ref{lem:strip-bound}, there are
\[
\ll Y^{\rho - 1}(H_Y + 1)
\]
possible values of $g'$. Therefore the number of quadruples
$(m, g, m', g') \in (P^* \times G)^2$ satisfying the badness condition is at
most
\[
\ll M^2|G|Y^{\rho - 1}(H_Y + 1)
\ll M^2Y^{2\rho - 1}(\log Y)^{\rho + 4}\,.
\]
Since every bad unordered pair in $A^*$ is counted by at least one such
choice of representatives, the total number of bad unordered
representations is bounded by the same quantity.

We now show that no two distinct good representations have the same sum.
Suppose that $\{x_1, x_2\}$ and $\{x_3, x_4\}$ are good unordered
representations satisfying
\[
x_1 + x_2 = x_3 + x_4\,.
\]
By interchanging the two representations and reordering their summands, we
may assume that $|x_1|_v$ is largest among the four absolute values and
that $|x_3|_v \ge |x_4|_v$. Since the representations are good,
\[
\lambda_v(x_1) - \lambda_v(x_2) > H_Y\,,
\qquad
\lambda_v(x_3) - \lambda_v(x_4) > H_Y\,.
\]
Exponentiating gives
\[
|x_2|_v < e^{-H_Y}|x_1|_v\,,
\qquad
|x_4|_v < e^{-H_Y}|x_3|_v \le e^{-H_Y}|x_1|_v\,.
\]
From $x_3 - x_1 = x_2 - x_4$ and the triangle inequality, we obtain
\begin{align*}
\left|\frac{x_3}{x_1} - 1\right|_v
&=
\left|\frac{x_2 - x_4}{x_1}\right|_v\\
&\le \frac{|x_2|_v + |x_4|_v}{|x_1|_v}\\
&< 2e^{-H_Y}\\
&= \frac{1}{2}\exp\left(-C_{\mathrm{sep}}(\log Y)^{\rho + 4}\right)\,.
\end{align*}
If $x_3 \ne x_1$, write the chosen factorizations as
\[
x_1 = m_1g_1\,,
\qquad
x_3 = m_3g_3\,.
\]
Lemma~\ref{lem:bounded-separation} gives
\[
\left|\frac{x_3}{x_1} - 1\right|_v
\ge
\exp\left(-C_{\mathrm{sep}}(\log Y)^{\rho + 4}\right)\,,
\]
a contradiction. Hence $x_1 = x_3$. Subtracting equal summands from the
original equality gives $x_2 = x_4$. Thus the two unordered representations
are identical.

For each $t \in F$, let $u_{\text{bad}}(t)$ and $u_{\text{good}}(t)$ count
the bad and good unordered representations of $t$ by elements of $A^*$,
respectively. Then
\[
u_{A^*}(t) = u_{\text{bad}}(t) + u_{\text{good}}(t)\,,
\]
and the preceding paragraph gives $u_{\text{good}}(t) \le 1$ for every $t$.
Therefore
\[
\max\{u_{A^*}(t) - 1,\, 0\} \le u_{\text{bad}}(t)\,.
\]
Summing over $t$, we get
\[
\Colun(A^*)
\le \sum_t u_{\text{bad}}(t)
\ll M^2Y^{2\rho - 1}(\log Y)^{\rho + 4}\,.
\]

If $0 \notin A$, then $A = A^*$ and we are done. Otherwise,
Lemma~\ref{lem:adjoin-zero} gives
\[
\Colun(A) \le \Colun(A^*) + |A^*| + 1\,.
\]
Since $M \ge 1$, $\rho \ge 1$, and $Y \ge 2$, we have
\[
|A^*| + 1 \ll MY^\rho \le M^2Y^{2\rho - 1}\,.
\]
The additional term is absorbed into the displayed collision bound.
\end{proof}

\subsection{Proof of the asymptotically maximal-sumset theorem}

\begin{proof}[Proof of Theorem~\ref{thm:maximal}]
By Lemma~\ref{lem:collision-identity},
\[
\frac{|A + A|}{|A|^2}
=
\frac{1}{2} + \frac{1}{2|A|} - \frac{\Colun(A)}{|A|^2}\,.
\]
By \eqref{eq:A-at-least-G},
\[
\frac{1}{|A|} \ll Y^{-\rho} \ll Y^{-1}\,.
\]
On the other hand, \eqref{eq:low-height-size} and
Proposition~\ref{prop:collision-bound} give
\[
\frac{\Colun(A)}{|A|^2}
\ll
\frac{M^2Y^{2\rho - 1}(\log Y)^{\rho + 4}}
{M^2Y^{2\rho}(\log Y)^{-2\rho}}
=
\frac{(\log Y)^{3\rho + 4}}{Y}\,.
\]
Taking $C = 3\rho + 4$ proves the theorem. All constants used above depend
only on the data listed in its statement, so the estimate is uniform in
the coefficient set, the exponent set, and the scale.
\end{proof}

\section{Product growth and the fixed-field obstruction}

\subsection{Factorizations of bounded algebraic integers}

We now prove the product-growth estimate needed for
Theorem~\ref{thm:main}. The relevant input is a lower bound for the
product set of the coefficients. We obtain it by bounding the number of
factorizations of a nonzero algebraic integer into two factors with
bounded conjugates.

For a nonzero integral ideal $\mathfrak a \subseteq \mathcal O_F$, let
$\tau_F(\mathfrak a)$ be the number of its integral ideal divisors. Thus
$\mathfrak b \mid \mathfrak a$ means that
$\mathfrak a = \mathfrak b\mathfrak c$ for an integral ideal
$\mathfrak c$. We use unique factorization of nonzero ideals into prime
ideals, as in \cite[Chapter~1]{Narkiewicz2004}.

\begin{lemma}[Ideal-divisor bound]\label{lem:ideal-divisors}
For every $\eta > 0$ and every nonzero integral ideal $\mathfrak a$,
\[
\tau_F(\mathfrak a) \ll_{F, \eta} (N\mathfrak a)^\eta\,.
\]
\end{lemma}

\begin{proof}
Write
\[
\mathfrak a = \prod_{\mathfrak q}\mathfrak q^{a_{\mathfrak q}}\,.
\]
Then
\[
\tau_F(\mathfrak a) = \prod_{\mathfrak q}(a_{\mathfrak q} + 1)\,.
\]
For a rational prime $p$, its exponent in $N\mathfrak a$ is
\[
k_p := \sum_{\mathfrak q \mid p}f_{\mathfrak q}a_{\mathfrak q}\,.
\]
There are at most $d$ prime ideals above $p$, and every
$a_{\mathfrak q} \le k_p$. Hence
\[
\prod_{\mathfrak q \mid p}(a_{\mathfrak q} + 1) \le (k_p + 1)^d\,,
\]
so, writing $\tau(m)$ for the ordinary divisor function,
\[
\tau_F(\mathfrak a) \le \tau(N\mathfrak a)^d\,.
\]

For completeness, $\tau(m) \ll_\theta m^\theta$ for every $\theta > 0$.
Indeed, if $p^\theta \ge 2$ and $a \ge 1$, then
\[
a + 1 \le 2^a \le p^{\theta a}\,.
\]
For each of the finitely many primes with $p^\theta < 2$, the quantity
$(a + 1)p^{-\theta a}$ is bounded as $a$ varies over the nonnegative
integers. Multiplying these finitely many bounds and then multiplying
over the prime factors of $m$ proves the assertion. Applying it with
$\theta = \eta / d$ gives the ideal-divisor bound.
\end{proof}

\begin{lemma}[Bounded integral
factorizations]\label{lem:integral-factorizations}
Let $T \ge 2$, and let $Q \subset \mathcal O_F \setminus \{0\}$ be finite,
with
\[
|\sigma(q)| \le T
\qquad(q \in Q, \ \sigma : F \hookrightarrow \mathbb C)\,.
\]
For $z \in F^\times$, define
\[
r_Q(z) := |\{(q_1, q_2) \in Q^2 : q_1q_2 = z\}|\,.
\]
For every $\eta > 0$,
\[
\max_{z \in F^\times}r_Q(z) \ll_{F, \eta} T^\eta\,.
\]
\end{lemma}

\begin{proof}
It suffices to consider $z \in QQ$. Then $z$ is a nonzero algebraic
integer, and every conjugate has modulus at most $T^2$. Thus
\[
N((z)) = |N_{F / \Q}(z)| \le T^{2d}\,.
\]
If $q_1q_2 = z$, then $(q_1)$ is an integral ideal divisor of $(z)$.
Applying Lemma~\ref{lem:ideal-divisors} with exponent $\eta / (4d)$ shows
that there are at most
\[
\ll_{F, \eta} N((z))^{\eta / (4d)} \le T^{\eta / 2}
\]
possible ideals $(q_1)$.

Fix one such ideal $\mathfrak b$ that occurs, and choose
$q_0 \in Q$ with $(q_0) = \mathfrak b$. If $q_1 \in Q$ has the same
principal ideal, then
\[
u := q_1 / q_0 \in \mathcal O_F^\times\,.
\]
By Lemma~\ref{lem:height-identities},
\[
h(u) \le h(q_1) + h(q_0) \le 2\log T\,.
\]
The unit group is fixed and finitely generated. If $r_F$ is its rank
modulo torsion, Lemma~\ref{lem:Gamma-height-count} gives
\[
|\{u \in \mathcal O_F^\times : h(u) \le 2\log T\}|
\ll_F (1 + \log T)^{r_F}\,.
\]
Every unit determines $q_1 = uq_0$ uniquely, and $q_1$ determines
$q_2 = z / q_1$ uniquely. Consequently,
\[
r_Q(z) \ll_{F, \eta} T^{\eta / 2}(1 + \log T)^{r_F}
\ll_{F, \eta} T^\eta\,.\qedhere
\]
\end{proof}

\begin{corollary}[Product growth of admissible
coefficients]\label{cor:coefficient-products}
For every admissible coefficient family and every $\varepsilon > 0$,
\[
|P_n^*P_n^*| \gg_\varepsilon M_n^{2 - \varepsilon}\,.
\]
In particular, $|P_n^*P_n^*| \ge M_n^{2 - o(1)}$ whenever $M_n \to \infty$.
The constants depend only on $F$, $C_P$, $\kappa$, and $\varepsilon$.
\end{corollary}

\begin{proof}
Use the normalization $Q_n = w_nP_n^*$ from
Definition~\ref{def:coefficients}, and put
\[
T_n := \max\{2,\, C_P\widehat M_n^\kappa\}
\le 2C_P\widehat M_n^\kappa\,.
\]
Applying Lemma~\ref{lem:integral-factorizations} with
$\eta = \varepsilon / \kappa$ gives
\[
\max_z r_{Q_n}(z) \ll_\varepsilon \widehat M_n^\varepsilon
\ll_\varepsilon M_n^\varepsilon\,.
\]
Since
\[
M_n^2 = \sum_{z \in Q_nQ_n}r_{Q_n}(z)\,,
\]
we obtain
\[
|Q_nQ_n| \gg_\varepsilon M_n^{2 - \varepsilon}\,.
\]
Finally, $Q_nQ_n = w_n^2(P_n^*P_n^*)$, so the two product sets have the
same cardinality. The stated asymptotic form follows by applying the
bound with an arbitrarily small fixed positive exponent.
\end{proof}

\subsection{Product saving and the relative scales}

\begin{lemma}[Scale forcing]\label{lem:scale-forcing}
Let $P$ be a member of an admissible coefficient family, put
$M := |P \setminus \{0\}|$, and let $G \subset F^\times$ be finite and
nonempty. Write $A := PG$ and $L := |G|$. For every fixed
$\delta \in (0, 1]$,
\[
|AA| \le |A|^{2 - \delta}
\quad \Longrightarrow \quad
M^{\delta / 2} \ll_\delta L^{2 - \delta}\,.
\]
If, in addition, $G \subset \Gamma$ satisfies
\eqref{eq:lattice-hypothesis} at scale $Y \ge 2$, then
\begin{equation}\label{eq:forced-scales}
M \ll_\delta Y^{2\rho(2 - \delta) / \delta}\,,
\qquad
R(P) \ll_\delta \log Y\,.
\end{equation}
\end{lemma}

\begin{proof}
Choose any $g_0 \in G$. Since $g_0^2P^*P^* \subseteq AA$,
Corollary~\ref{cor:coefficient-products}, with
$\varepsilon = \delta / 2$, gives
\[
|AA| \ge |P^*P^*| \gg_\delta M^{2 - \delta / 2}\,.
\]
On the other hand, $|P| \le M + 1 \le 2M$, so
\[
|A| \le |P||G| \le 2ML\,.
\]
Under the assumed product saving, it follows that
\[
M^{2 - \delta / 2}
\ll_\delta |AA|
\le |A|^{2 - \delta}
\le (2ML)^{2 - \delta}\,.
\]
Dividing by $M^{2 - \delta}$ proves the first assertion.

For a full-dimensional multiplicative lattice set, $L \ll Y^\rho$.
Therefore
\[
M \ll_\delta Y^{2\rho(2 - \delta) / \delta}\,.
\]
Since $Y \ge 2$, this gives
$\log\max\{2,\, M\} \ll_\delta \log Y$.
Lemma~\ref{lem:coefficient-height} now proves the second assertion in
\eqref{eq:forced-scales}.
\end{proof}

\subsection{Proof of the main theorem}

\begin{proof}[Proof of Theorem~\ref{thm:main}]
To prove \eqref{eq:main-maximal}, fix $\delta \in (0, 1]$ and suppose that
\[
|A_nA_n| \le |A_n|^{2 - \delta}
\]
for all sufficiently large $n$. By
Lemma~\ref{lem:scale-forcing},
\[
M_n \ll_\delta Y_n^{2\rho(2 - \delta) / \delta}\,,
\qquad
R(P_n) \ll_\delta \log Y_n\,.
\]
Since $|G_n| \ll Y_n^\rho$, we also have
\[
|A_n| \le (M_n + 1)|G_n|
\ll_\delta Y_n^{2\rho(2 - \delta) / \delta + \rho}\,.
\]
The hypothesis $|A_n| \to \infty$ therefore implies $Y_n \to \infty$.
The hypotheses of Theorem~\ref{thm:maximal} now hold with a fixed
ratio-height constant that may depend on $\delta$. Hence
\[
\frac{|A_n + A_n|}{|A_n|^2}
=
\frac{1}{2} + O_\delta\left(\frac{(\log Y_n)^C}{Y_n}\right)\,,
\]
where $C = 3\rho + 4$ is permissible. This proves
\eqref{eq:main-maximal}.

To prove \eqref{eq:main-growth}, suppose for a contradiction that there
exists a fixed $\varepsilon \in (0, 1)$ and an infinite subsequence along
which
\[
\max\{|A_n + A_n|,\, |A_nA_n|\} < |A_n|^{2 - \varepsilon}\,.
\]
Along this subsequence, the original hypotheses remain valid, including
$|A_n| \to \infty$. In particular,
\[
|A_nA_n| \le |A_n|^{2 - \varepsilon}\,.
\]
Applying the first part of the proof to this subsequence with
$\delta = \varepsilon$ gives $Y_n \to \infty$ and
\[
\frac{|A_n + A_n|}{|A_n|^2} \to \frac{1}{2}\,.
\]
However, the assumed bound on the sumset gives
\[
\frac{|A_n + A_n|}{|A_n|^2}
< |A_n|^{-\varepsilon} \to 0\,,
\]
a contradiction. Therefore, for every fixed $\varepsilon \in (0, 1)$,
the maximum is at least $|A_n|^{2 - \varepsilon}$ for all sufficiently
large $n$, which is precisely \eqref{eq:main-growth}.
\end{proof}

The same maximal-sumset conclusion holds if the product-saving
hypothesis is weakened to
\[
|A_nA_n| \ll |A_n|^{2 - \delta}
\]
for a fixed $\delta \in (0, 1]$. Since $|A_n| \to \infty$, the implied
constant is at most $|A_n|^{\delta / 2}$ for all sufficiently large $n$.
Thus $|A_nA_n| \le |A_n|^{2 - \delta / 2}$ eventually, and
Theorem~\ref{thm:main} applies with $\delta / 2$.

\section{Applications}

We now record some families of interest to which the results proved in Sections
3 and 4 apply.

\subsection{Fixed-field additive and unit boxes}

We verify that the boxes in Corollary~\ref{cor:BSSZ} satisfy the two
admissibility conditions.

\begin{lemma}[Admissibility of the fixed-field
boxes]\label{lem:box-admissibility}
Let $F$ be totally real of degree $d \ge 2$, and let
$\varepsilon \in (0, 1)$ be fixed. For $X \in \Zpos$ and $Y \ge 2$ real, put
\[
P_X := X + B_F^+(\varepsilon X)\,,
\qquad
G_Y := B_F^\times(Y)\,.
\]
Then
\[
|P_X| \asymp_{F, \varepsilon} X^d\,,
\qquad
|G_Y| \asymp_F Y^{d - 1}\,.
\]
The coefficient sets satisfy \eqref{eq:coefficient-hypothesis} with
$w_X = 1$ and $\kappa = 1 / d$. The unit sets satisfy
\eqref{eq:lattice-hypothesis} for the fixed group
$\Gamma = \mathcal O_F^\times$, of rank $\rho = d - 1$.
\end{lemma}

\begin{proof}
Choose an integral basis $\omega_1, \ldots, \omega_d$ of $\mathcal O_F$,
and let
\[
M_+ := (\sigma_i(\omega_j))_{1 \le i, j \le d}\,.
\]
The Minkowski embedding identifies $\mathcal O_F$ with the full lattice
$M_+\Z^d \subset \mathbb R^d$. In particular, $M_+$ is invertible. Put
\[
\Omega_+ := \{\mathbf x \in \mathbb R^d : \|M_+\mathbf x\|_\infty \le 1\}\,.
\]
Then
\[
B_F^+(\varepsilon X)
=
\left\{\sum_{j = 1}^d z_j\omega_j :
\mathbf z \in \Z^d \cap \varepsilon X\Omega_+\right\}\,.
\]
Since $M_+$ and its inverse are fixed linear maps, there are constants
$0 < c < C$ such that
\[
[-c, c]^d \subseteq \varepsilon\Omega_+ \subseteq [-C, C]^d\,.
\]
Counting integer points in the dilated bounding cubes gives
\[
|P_X| = |B_F^+(\varepsilon X)| \asymp_{F, \varepsilon} X^d
\qquad(X \ge 1)\,.
\]

Since $X$ is a positive integer, each $m \in P_X$ is an algebraic
integer and satisfies
\[
(1 - \varepsilon)X \le \sigma_i(m) \le (1 + \varepsilon)X
\qquad(1 \le i \le d)\,.
\]
Thus $0 \notin P_X$, and every conjugate is at most
$(1 + \varepsilon)X \ll_{F, \varepsilon} |P_X|^{1 / d}$. This proves the
coefficient
condition.

For the multiplicative factor, the $S$-unit theorem with $S$ consisting
only of the archimedean places is Dirichlet's unit theorem. It identifies
$\mathcal O_F^\times / \{\pm1\}$ with a rank-$(d - 1)$ lattice in
\[
\mathcal H := \{\mathbf x \in \mathbb R^d : x_1 + \cdots + x_d = 0\}\,.
\]
Choose fundamental units $\varepsilon_1, \ldots, \varepsilon_{d - 1}$
whose images under $\psi \circ \pi$ are the standard basis vectors of
$\Z^{d - 1}$, and put
\[
M_\times
:= \bigl(\log|\sigma_i(\varepsilon_j)|\bigr)_
{1 \le i \le d, \ 1 \le j \le d - 1}\,.
\]
This matrix has full column rank. Every unit has a unique expression
\[
u = \xi\prod_{j = 1}^{d - 1}\varepsilon_j^{z_j}\,,
\qquad
\xi \in \{\pm1\}\,,
\qquad
\mathbf z \in \Z^{d - 1}\,,
\]
and its logarithmic vector is $M_\times \mathbf z$. Therefore, with
\[
\Omega_\times
:= \{\mathbf x \in \mathbb R^{d - 1} :
\|M_\times \mathbf x\|_\infty \le 1\}\,,
\]
we have the exact description
\[
G_Y
=
\left\{\xi\prod_{j = 1}^{d - 1}\varepsilon_j^{z_j} :
\xi \in \{\pm1\}\,, \
\mathbf z \in \Z^{d - 1} \cap Y\Omega_\times \right\}\,.
\]
The linear map $M_\times$ is injective, so $\Omega_\times$ is bounded
and contains a neighborhood of zero. Comparison with two fixed cubes
now gives
\[
|G_Y| = 2|\Z^{d - 1} \cap Y\Omega_\times | \asymp_F Y^{d - 1}\,,
\]
and all exponent coordinates have absolute value $O_{F, \psi}(Y)$. These are
exactly the conditions in \eqref{eq:lattice-hypothesis} for
$\rho = d - 1$.
\end{proof}

\begin{proof}[Proof of Corollary~\ref{cor:BSSZ}]
Apply Lemma~\ref{lem:box-admissibility} to $P_{X_n}$ and $G_{Y_n}$.
All admissibility constants depend only on the fixed field, the chosen map
$\psi$, and
$\varepsilon$, not on $X_n$ or $Y_n$. Theorem~\ref{thm:main} therefore
applies to every sequence in the statement.
\end{proof}

The same proof permits families $\widetilde P_X \subseteq P_X$ satisfying
$|\widetilde P_X| \gg X^\beta$ for a fixed $\beta > 0$, and families
$\widetilde G_Y \subseteq B_F^\times(CY)$ satisfying
$|\widetilde G_Y| \gg Y^{d - 1}$ for a fixed $C > 0$. The former satisfy the
coefficient condition with $\kappa = 1 / \beta$, and the latter retain the
required exponent bound and cardinality.

Returning to the full boxes, without a product saving the sumset need
not be maximal.
For example, if $Y$ is fixed and $X \to \infty$, then $G_Y$ is fixed and
\[
A_{X, Y} + A_{X, Y}
\subseteq B_F^+\bigl(2(1 + \varepsilon)e^Y X\bigr)\,.
\]
The same lattice count gives
$|A_{X, Y} + A_{X, Y}| \ll_{F, \varepsilon, Y} X^d \asymp_{F, \varepsilon, Y}
|A_{X, Y}|$.
In this regime, however, Corollary~\ref{cor:coefficient-products} gives
$|A_{X, Y}A_{X, Y}| \ge |A_{X, Y}|^{2 - o(1)}$. A fixed product saving forces
the other regime, where the unit scale tends to infinity and the sumset
is asymptotically maximal.

\subsection{Arithmetic--geometric progression products}

We finish with a family formed from the products of arithmetic and geometric
progressions of common length $n$. Let
$K \in \Zpos$ be fixed, and write $[n] := \{1, 2, \ldots, n\}$. For each
$k \in [K]$, let $a_k \in F$ and $d_k, r_k \in F^\times$ be fixed. Define
\[
P_n^k := \{a_k + (i - 1)d_k : i \in [n]\}\,,
\qquad
G_n^k := \{r_k^j : j \in [n]\}\,,
\]
then
\[
P_n := P_n^1 + \cdots + P_n^K\,,
\qquad
G_n := G_n^1\cdots G_n^K\,,
\]
and finally $A_n := P_nG_n$.

\begin{corollary}[Arithmetic--geometric progression
products]\label{cor:progressions}
If at least one $r_k$ is not a root of unity, then there exists a constant
$C > 0$, depending only on the fixed field and progression parameters,
such that
\[
\frac{|A_n + A_n|}{|A_n|^2}
=
\frac{1}{2} + O\left(\frac{(\log n)^C}{n}\right)
\qquad(n \to \infty)\,.
\]
\end{corollary}

\begin{proof}
For this application, take $\Gamma := \langle r_1, \ldots, r_K\rangle$.
The hypothesis implies that its rank $\rho$ modulo torsion is positive.
Under an identification $\psi : \overline{\Gamma} \to \Z^\rho$, put
\[
\mathbf v_k := \psi(\pi(r_k))\,.
\]
The exponent image of $G_n$ is
\[
\left\{\sum_{k = 1}^K j_k\mathbf v_k : j_k \in [n]\right\}\,,
\]
which is contained in a box of side lengths $O(n)$. The vectors
$\mathbf v_k$ span a rank-$\rho$ lattice. Choose $\rho$ of them that are
linearly independent over $\Q$, vary their indices over $[n]$, and hold
the remaining indices fixed. This gives at least $n^\rho$ distinct
exponent vectors. Hence $G_n$ satisfies the multiplicative lattice
conditions at scale $Y = n$.

Choose a fixed positive rational integer $W$ such that
\[
Wa_k, Wd_k \in \mathcal O_F
\qquad(1 \le k \le K)\,.
\]
For $m \in P_n$, the element $Wm$ is an algebraic integer, and every
archimedean conjugate of
\[
Wm = W\sum_{k = 1}^K a_k + W\sum_{k = 1}^K d_kt_k\,,
\qquad
t_k \in \{0, 1, \ldots, n - 1\}\quad(1 \le k \le K)\,,
\]
has absolute value $O(n)$. Lemma~\ref{lem:height-identities} gives
$h(Wm) \ll \log n$ for $m \ne 0$ and $n \ge 2$. Thus, for
$m, m' \in P_n^*$,
\[
h(m / m') = h\left(\frac{Wm}{Wm'}\right)
\le h(Wm) + h(Wm') \ll \log n\,.
\]
Moreover, $P_n$ contains at least $n$ distinct elements by varying a
single arithmetic index, so $P_n^* \ne \varnothing$ for $n \ge 2$.
Theorem~\ref{thm:maximal}, with $Y = n$, now proves the result.
\end{proof}

\section{Statements and disclosures}
\label{sec:disclosures}
\subsection{Use of generative AI}

Generative AI is a substantial part of the authors' research process. The
authors
made use of OpenAI's GPT-5.6 Sol, GPT-6 Astra, and GPT-6.1 Sol, and Anthropic's
Claude Fable 5, Fable 5.1, and Opus 5.5 in the development of this paper.

These models were used to assist with literature search and review, explain and
analyze papers, assist
with drafting and editing, discuss proof ideas, suggest lemmas and consequences,
outline proofs, and develop computational tests of results. They were also used
substantively during the development of proofs, in some cases proposing key
arguments or generating candidate proofs in response to mathematical problems
posed by the authors.

\subsection{Statement of responsibility}

The authors take full responsibility for the contents of the paper,
including its mathematical claims, arguments, and citations.

\section*{Acknowledgements}
We thank Professor Ilya Shkredov of Purdue University for his continued support
throughout the development of this work, including posing the original problem
we sought to solve, making insightful comments on article drafts, providing
useful discussions, and inviting us to present our work at an informal seminar
at Purdue in October, 2026. We also thank the attendees of the talk for giving
us feedback and ideas related to our work.

\printbibliography

\end{document}